\documentclass{IET}

\usepackage{amsmath, amsthm, amscd, amsfonts, amssymb, graphicx, color,  mathrsfs}
\newtheorem{theorem}{Theorem}

\newtheorem{remark}{Remark}
\newtheorem{lemma}{Lemma}

\newtheorem{definition}{Definition}
\DeclareMathAlphabet{\mathpzc}{OT1}{pzc}{m}{it}

\begin{document}

\title{Sharp  estimates for the Schr\"odinger equation associated with the twisted Laplacian}

\author{Duv\'an Cardona}
\affil{Pontificia Universidad Javeriana, Department of Mathematics, Bogot\'a- Colombia.\vspace{0.2cm}}
\affil{Current address: Ghent University, Department of Mathematics: Analysis, Logic and Discrete Mathematics,  Krijgslaan 281, S8
9000 Gent-Belgium\vspace{0.2cm}\\}
\affil[1]{duvanc306@gmail.com}
\abstract{In this note  we obtain  some Strichartz estimates for the Schr\"odinger equation associated with the  twisted Laplacian on $\mathbb{C}^{n}\cong \mathbb{R}^{2n}.$   The initial data will be considered in suitable Sobolev spaces associated to the twisted Laplacian. \\{\it The final version of this paper will appear in Rep. Math. Phys.}  }

\maketitle
\section{Introduction}
In this note we provide sharp Strichartz estimates for  the Schr\"odinger equation associated with the twisted Laplacian. Although we are focused in a typical problem of PDE, that is, to investigate the well-posedness for the Schr\"odinger equation \eqref{SEq},  the novelty of our paper is that we exploit the spectral properties  of the twisted Laplacian emerging from the seminal work of Koch and  Ricci \cite{KochRicci2007}. In order to present our main result, let us recall some fundamental notions.

Let us consider the  twisted Laplacian on $\mathbb{R}^{2n}$ (introduced by R. Strichartz in \cite{Twisted})
\begin{equation}
    \mathcal{L}=-\frac{1}{2}\Delta_{\mathfrak{X}},\,\,\Delta_{\mathfrak{X}}:=\sum_{j=1}^n(X_{j}^2+Y_{j}^2),
\end{equation}
 where $ \Delta_{\mathfrak{X}}$ is the sums of squares associated to the family $\mathfrak{X}=\{ X_j,Y_j :1\leq j\leq n\},$  $X_{j}=\partial_{x_j}+iy_j,$ $Y_{j}=\partial_{y_j}-ix_j,$ $i^2=-1,$ and $(x,y)\in \mathbb{R}^{2n}.$ In quantum mechanics,  the twisted Laplacian comes up as the Hamiltonian describing the motion of a charged particle $\textbf{q}$ in a uniform magnetic field $\textbf{B}=B\hat{\textbf{k}},$ $\hat{\textbf{k}}=(0,0,1)\in \mathbb{R}^3,$ perpendicular to the plane $x$-$y,$ (see the seminal work of Landau \cite{Landau1930} or Landau and  Lifshitz \cite[Chapter XV]{LandauLifshitz}). Moreover, for $n=1,$ the eigenvalues of $\mathcal{L}$ are known as Landau levels, which are the corresponding energy levels of the charged particle $\textbf{q}$ (see Landau \cite{Landau1930}),  and the
corresponding eigenfunctions of $\mathcal{L}$ are given by the Wigner transforms of the Hermite functions (see e.g. Molahajloo, Rodino and Wong \cite[pag. 78]{Molahajloo}).

In this note we obtain sharp Strichartz estimates for the Schr\"odinger equation associated with $\mathcal{L}$ which is  given by
\begin{equation}\label{SEq}
iu_t(t,z)+\frac{1}{2}\Delta_{\mathfrak{X}}u(t,z)=0,\,\,z=(x,y), 
\end{equation}
with initial data $u(0,\cdot)=f$ defined on $\mathbb{R}^{2n}.$  Estimates for this model and other relationed with the Landau Hamiltonian have been extensively developed in the references Ruzhansky and Tokmagambetov \cite{Ruz-Tok,Ruz-Tok',Ruz-Tok-1},  P.K.  Ratnakumar \cite{Rat2008}, P.K. Ratnakumar and V. K. Sohani\cite{RatSoha2013,RatSoha2015}, J. Tie\cite{Tie} and Z. Zhang and S. Zheng \cite{ZhangZheng2010}. The  obtained results in some of these references are analogues of the well known regularity properties for the  classical Schr\"odinger equation
\begin{equation}\label{SEq'}
iu_t(t,x)+\Delta_x u(t,x)=0, 
\end{equation}
investigated in the classical works of J. Ginibre and G. Velo \cite{GinVel},   A. Moyua and L. Vega \cite{MoyuaVega}, M. Keel and T. Tao \cite{KeelTao} and references therein. 

Let us observe that in view of the identity
\begin{equation}
-\Delta_{\mathfrak{X}}=-\Delta_{(z)}+|z|^2,\,\,\Delta_{(z)}:=\Delta_{x,y}+i(y,-x)\cdot (\nabla_x,\nabla_y )
\end{equation} where $\Delta_{x,y}$ is the Laplacian on $\mathbb{R}^{2n},$ $\nabla_x$ is the standard gradient field, and taking into account that the quantum harmonic oscillator is given by $H=-\Delta_x+|x|^2,$ the twisted Laplacian $\mathcal{L}=-\frac{1}{2}\Delta_{ \mathcal{L} }$ is sometimes called the quantum complex harmonic oscillator.

Taking into account the following estimate  proved by P.K.  Ratnakumar \cite{Rat2008}  
    \begin{equation}\label{rat2008}
     \Vert u(t,z)\Vert_{L^q_{t}([0,2\pi],\,L^p(\mathbb{R}^{2n}))}\leq C\Vert f\Vert_{L^{2}(\mathbb{R}^{2n})},
    \end{equation}  for $2<q<\infty$ and $\frac{1}{q}\geq n(\frac{1}{2}-\frac{1}{p})$ or $1\leq q\leq  2$ and $2\leq p<\frac{2n}{n-1},$ $n\geq 1$, we will consider the initial data $f$ in Sobolev spaces associated to $\mathcal{L}$ but modeled on $L^2$. It is important to mention that in P.K. Ratnakumar and V.K. Sohani \cite{RatSoha2013,RatSoha2015} the previous Ratnakumar result was extended to the non-homogeneous (and possibly non-linear) Sch\"odinger equation associated to $\mathcal{L}$. Also,
   the Schr\"odinger equation associated to $\mathcal{L},$ but with non-linear  polynomials terms can be found in Z. Zhang and S. Zheng  \cite{ZhangZheng2010}.

In order to motivate our main result, let us consider  the Schr\"odinger equation associated to the harmonic oscillator:
\begin{equation}
    iu_t(t,x)-Hu(x,t)=0,\,\,\,u(0,\cdot)=f.
\end{equation} The main result  in B. Bongioanni and  K.M. Rogers, \cite{BonRog} is the following sharp theorem: for  $\frac{2(d+2)}{n}\leq p\leq \infty,$ and $2\leq q< \infty$ with $\frac{1}{q}\leq \frac{d}{2}(\frac{1}{2}-\frac{1}{p}),$  
\begin{equation}\label{BV}
\Vert u(t,x)\Vert_{L^p_{x}(\mathbb{R}^d\,, L^q_{t}[0,2\pi])}\leq C_s\Vert f\Vert_{\mathcal{H}^s(\mathbb{R}^d)}
\end{equation}
holds true for all $s\geq s_{d,p,q}:=d(\frac{1}{2}-\frac{1}{p})-\frac{2}{q}.$ If $s<s_{d,p,q}$ then \eqref{BV} is false.  
 In the result above $\mathcal{H}^s$ is the Sobolev space associated to $H$ and with norm $\Vert f\Vert_{\mathcal{H}^s}:=\Vert H^{\frac{s}{2}}f \Vert_{L^2}.$ The proof of  \eqref{BV} involves Strichartz estimates of M. Keel and T. Tao \cite{KeelTao} and
 Wainger's Sobolev embedding theorem.
 In terms of the Sobolev space $W^{s,H}(\mathbb{R}^d)$ defined by the norm $\Vert f\Vert_{W^{s,H}(\mathbb{R}^d)}=\Vert H^sf\Vert_{L^2}=\Vert f\Vert_{\mathcal{H}^{2s}},$ \eqref{BV} can be formulated as
 \begin{equation}
 \Vert u(t,x)\Vert_{L^p_{x}(\mathbb{R}^d\,, L^q_{t}[0,2\pi])}\leq C_s\Vert f\Vert_{W^{s,H}(\mathbb{R}^d)}
\end{equation}
for all $s\geq s'_{d,p,q}:=\frac{d}{2}(\frac{1}{2}-\frac{1}{p})-\frac{1}{q}.$ Because, the twisted Laplacian acts on complex functions in $\mathbb{R}^d,$ $d=2n,$ we expect to obtain  similar results to \eqref{BV}, but, on Sobolev spaces $W^{s,\mathcal{L}}$ associated to $\mathcal{L},$ when $s\geq \frac{d}{2}(\frac{1}{2}-\frac{1}{p})-\frac{1}{q}=n(\frac{1}{2}-\frac{1}{p})-\frac{1}{q}.$
  Our main theorem can be announced as follows.

 \begin{theorem}\label{MAINtHEO'} Let us assume $n\geq 1,$  $2\leq q<\infty,$ and $2\leq p\leq \infty.$ Then, the following estimate
\begin{equation}\label{sobolecestimate'}
\Vert u(t,z) \Vert_{L^{p}_z[\mathbb{R}^{2n},\,L^q_t[0,2\pi]]}\leq C\Vert f\Vert_{W^{s,\,\mathcal{L}}(\mathbb{R}^{2n})} \end{equation}
holds true for every $s\geq \varkappa_{p,q},$ where $\varkappa_{p,q}:=\frac{1}{2}(\frac{1}{2}+\frac{1}{p})-\frac{1}{q}$ for $2\leq p\leq \frac{4n+2}{2n-1}$ and $\varkappa_{p,q}:=n(\frac{1}{2}-\frac{1}{p})-\frac{1}{q}$ for $ \frac{4n+2}{2n-1}\leq p\leq \infty.$  For $q=2,$ the Strichartz estimate \eqref{sobolecestimate'} is sharp in the sense that this inequality is the best possible.
 \end{theorem}
 
\begin{remark}\label{r1}
 If $2\leq p\leq \frac{4n+2}{2n-1}$ then $\frac{n}{2n+1}-\frac{1}{q}\leq \varkappa_{p,q}\leq \frac{1}{2}-\frac{1}{q}.$ If we assume $2\leq q\leq 2+\frac{1}{n},$  then $\frac{n}{2n+1}-\frac{1}{q}\leq 0$. So, in the interval $I_p=[2, \frac{4n+2}{2n-1}]$ containing  $p,$ the parameter $\varkappa_{p,q}$ admits non-positive values. Moreover, in this case $\varkappa_{p,q}\leq 0$  for $2\leq q\leq  2(\frac{1}{2}+\frac{1}{p})^{-1}.$ Similarly, $ \varkappa_{p,q}< 0$  if and only if  $\frac{n}{2n+1}-\frac{1}{q}<0,$ which in turn  is equivalent to the condition $2\leq q<  2(\frac{1}{2}+\frac{1}{p})^{-1}$.
 
 On the other hand, if $\frac{4n+2}{2n-1}\leq p\leq \infty$ then $\frac{n}{2n+1}-\frac{1}{q}\leq \varkappa_{p,q}\leq \frac{n}{2}-\frac{1}{q}.$ So, if we fix $2\leq q\leq 2+\frac{1}{n}, $ then $\frac{n}{2n+1}-\frac{1}{q}\leq 0$ and we deduce that $\varkappa_{p,q}$ admits non-positive  values. Let us note that $\varkappa_{p,q}=0$ when $\frac{1}{q}=n(\frac{1}{2}-\frac{1}{p})$ and $\varkappa_{p,q}<0$ for those parameters $p,q$ satisfying $\frac{1}{q}>n(\frac{1}{2}-\frac{1}{p}).$ This analysis shows that our main theorem admits different values of $p,q$ satisfying $\varkappa_{p,q}=0$ and $\varkappa_{p,q}<0.$ In view of the embedding $L^2\hookrightarrow W^{\varkappa,\,\mathcal{L}}$ for $\varkappa<0,$ we have showed that the usual condition $f\in L^{2}$ can be improved if we  consider those cases mentioned above where $\varkappa_{p,q}<0.$
\end{remark}

\begin{remark}
 The seminal work by Keel and  Tao \cite{KeelTao} shows that   the following Strichartz estimate
\begin{equation}\label{KT'}
\Vert v(t,x)\Vert_{L^q[(0,\infty), L_x^p(\mathbb{R}^d)]}\leq C\Vert g\Vert_{L^2(\mathbb{R}^n)},
\end{equation}
for the Schr\"odinger equation associated to the Laplacian $\Delta_{x}$ on $\mathbb{R}^n,$ given by
\begin{equation}
    iv_t(t,x)+\Delta_{x}v(x,t)=0,\,\,\,v(0,\cdot)=g,
\end{equation}
 holds true, if and only if $n=1$ and $2\leq p\leq \infty,$ $n=2$ and $2\leq p<\infty$ and $2\leq p<\frac{2n}{n-2}$ for $n\geq 3.$ Again, in view of the embedding $L^2\hookrightarrow W^{\varkappa,\,\mathcal{L}}$ for $\varkappa<0,$ we observe that the sharp situation $g\in L^{2}$ can be improved in our model,  by considering $f\in W^{\varkappa_{p,q},\,\mathcal{L}},$ if we analyse those cases in Remark \ref{r1} where $\varkappa_{p,q}<0.$

\end{remark}

 The plan of this paper is as follows. In the Section \ref{Preliminaries} we discuss some spectral properties of the twisted Laplacian. Finally, in Section \ref{Regularity properties} we prove our main theorem.

\section{Spectral decomposition of the twisted Laplacian }\label{Preliminaries}
In this section we provide some known facts on the spectrum of the twisted Laplacian and its relation with the harmonic oscillator. We will follow Thangavelu \cite{Thangavelu} for the spectral analysis of the twisted Laplacian.
Let  $H=-\Delta+|x|^2$ be the Hermite operator or (quantum) \emph{harmonic oscillator}. This operator extends to an unbounded self-adjoint operator on $L^{2}(\mathbb{R}^n) $, and its spectrum consists of the discrete set, $\lambda_\nu:=2|\nu|+n,$ $\nu\in \mathbb{N}_0^n,$  with a set of \emph{real eigenfunctions} $\phi_\nu, $ $\nu\in \mathbb{N}_0^n,$ (called Hermite functions) which provide an orthonormal basis of ${L}^2(\mathbb{R}^n).$
 Every Hermite function  $\phi_{\nu}$  on $\mathbb{R}^n$ has the form
\begin{equation}
\phi_\nu=\Pi_{j=1}^n\phi_{\nu_j},\,\,\, \phi_{\nu_j}(x_j)=(2^{\nu_j}\nu_j!\sqrt{\pi})^{-\frac{1}{2}}H_{\nu_j}(x_j)e^{-\frac{1}{2}x_j^2}
\end{equation}
where $x=(x_1,\cdots,x_n)\in\mathbb{R}^n$, $\nu=(\nu_1,\cdots,\nu_n)\in\mathbb{N}^n_0,$ and $$H_{\nu_j}(x_j):=(-1)^{\nu_j}e^{x_j^2}\frac{d^k}{dx_{j}^k}(e^{-x_j^2})$$ denotes the Hermite polynomial of order $\nu_j.$  By the spectral theorem, for every $f\in\mathscr{D}(\mathbb{R}^n)$ we have
\begin{equation}
Hf(x)=\sum_{\nu\in\mathbb{N}^n_0}\lambda_\nu\widehat{f}(\phi_\nu)\phi_\nu(x),\,\,\,
\end{equation} where $\widehat{f}(\phi_v) $ is the Hermite-Fourier transform of $f$ at $\nu$ defined by
\begin{equation}\label{ip} \widehat{f}(\phi_\nu) :=\langle f,\phi_\nu \rangle_{L^2(\mathbb{R}^n)}=\int_{\mathbb{R}^n}f(x)\phi_\nu(x)\,dx,\end{equation} where we omit the complex conjugate of $\phi_\nu(x),$ in the inner-product \eqref{ip},  because these functions are real-valued.
For every $\mu,\nu\in \mathbb{N}_0^n,$ the special Hermite function $\phi_{\mu\nu}$ is defined by
\begin{equation}
    \phi_{\mu\nu}(z)=(2\pi)^{-\frac{n}{2}}\int_{\mathbb{R}^n}e^{ix\cdot \xi}\phi_{\mu}(\xi+\frac{1}{2}y)\phi_{\nu}(\xi-\frac{1}{2}y)d\xi,\,\,\,z=(x,y)\in \mathbb{R}^{2n}.
\end{equation} A direct computation shows that
\begin{equation}
    \mathcal{L}\phi_{\mu\nu}=(2|\nu|+n)\phi_{\mu\nu}.
\end{equation} The  eigenvalues  $\lambda_\nu:=(2|\nu|+n), \,\nu\in \mathbb{N}_0^n,$ are the Laudau levels associated  to the  eigenfunctions $\phi_{\mu\nu}.$ It can be proved that  they also form a complete
orthonormal system in $L^2(\mathbb{R}^{2n}).$ Consequently, every function $f\in L^2(\mathbb{R}^{2n})$ has an expansion in terms of the special  Hermite functions given by
\begin{equation}
f=\sum_{\mu,\nu}\langle f, \phi_{\mu\nu}\rangle \phi_{\mu\nu}.
\end{equation} Newly, the spectral theorem applied to $\mathcal{L}$ gives
\begin{equation}
    \mathcal{L}f=\sum_{\mu,\nu} (2|\nu|+n)\widehat{f}(\phi_{\mu\nu})  \phi_{\mu\nu},\,\, \widehat{f}(\phi_{\mu\nu}):=\langle f, \phi_{\mu\nu}\rangle_{L^2(\mathbb{R}^{2n})},
\end{equation} for every $f\in \textnormal{Dom}(\mathcal{L}).$

The main tool in the harmonic analysis of the twisted Laplacian is the special Hermite semigroup, which we introduce as follows.  If $P_{\ell},$ $\ell\in2\mathbb{N}_0+n,$ is the spectral projection on $L^{2}(\mathbb{R}^{2n})$ given by
\begin{equation}
P_{\ell}f(z):=\sum_{\mu, 2|\nu|+n=\ell}\widehat{f}(\phi_{\mu\nu})\phi_{\mu\nu}(z),\,\,z=(x,y)\in\mathbb{R}^{2n},
\end{equation}
 then,  the special  Hermite semigroup (semigroup associated to the twisted Laplacian) $T_{t}:=e^{-t\mathcal{L}},$ $t>0$ is given by
 \begin{equation}
 e^{-t\mathcal{L}}f(z)=\sum_{\ell}e^{-t\ell}P_{\ell}f(z).
 \end{equation}
 By the functional calculus, and taking into account that the operator $\mathcal{L}$ is positive, the power $\mathcal{L}^s$,  $s\in \mathbb{R},$ of $\mathcal{L},$ can be defined in terms of the spectral projections $P_{\ell},$ $\ell\in 2\mathbb{N}_0+1,$ by
 \begin{equation}
 \mathcal{L}^sf(z)=\sum_{\ell}\ell^sP_{\ell}f(z),
 \end{equation} for every  $f\in \textnormal{Dom}(\mathcal{L}^s).$ Because we want to study the Schr\"odinger equation associated to $\mathcal{L}$ it is natural to consider for our analysis the Sobolev spaces associated to $\mathcal{L}$ which can be defined as follow.
 \begin{definition} Let $s\in \mathbb{R}.$ The $\mathcal{L}$-Sobolev space $W^{s,\,\mathcal{L}}(\mathbb{R}^{2n})$ consists of the completion of $\textnormal{Dom}(\mathcal{L}^s)$ (which is defined by those   functions $f$ satisfying  $\mathcal{L}^sf\in L^2(\mathbb{R}^{2n})$) with respect to the norm 
 \begin{equation}
     \Vert f\Vert_{W^{s,\,\mathcal{L}}(\mathbb{R}^{2n})}:=\Vert \mathcal{L}^s(f) \Vert_{L^2(\mathbb{R}^{2n})}=\left(\sum_{\ell}\ell^{2s}\Vert P_\ell f \Vert^2_{L^2(\mathbb{R}^{2n})}\right)^{\frac{1}{2}},
 \end{equation}for every $f\in \textnormal{Dom}(\mathcal{L}^s).$
 
 \end{definition}

 In this paper we also estimate the mixed norms $L^p_x(L^q_t)$ of  solutions to Sch\"rodinger equations by using the following version of Triebel-Lizorkin spaces associated to $\mathcal{L}$.
 \begin{definition}
Let us consider $1\leq p\leq \infty,$ $r\in\mathbb{R}$ and $0<q\leq \infty.$ The Triebel-Lizorkin space $\mathpzc{F}^r_{p,q}(\mathbb{R}^{2n}),$ associated to $\mathcal{L},$ to the family of projections $P_{\ell},$ $\ell \in 2\mathbb{N}+n,$ and to the parameters $p,q$ and $r$ is defined by those complex functions $f$ on $\mathbb{R}^{2n},$ satisfying
\begin{equation}
\Vert f\Vert_{\mathpzc{F}^r_{p,q}(\mathbb{R}^{2n})}:=\left\Vert \left(\sum_\ell \ell^{rq}|P_\ell f(z)|^{q}\right)^{\frac{1}{q}}\right\Vert_{L^p(\mathbb{R}_{z}^{2n})}<\infty.
\end{equation}
 \end{definition}
  The following are  embedding properties of such spaces and they can be proved from the definition. Because the proof is verbatim to one given for Triebel-Lizorkin spaces on $\mathbb{R}^n,$ we refer  the reader to Triebel \cite{Triebel1983,Triebel2006} and Lizorkin \cite{Lizorkin} for details.  
 \begin{itemize}
\item[(1)] ${\mathpzc{F}}^{r+\varepsilon}_{p,q_1}\hookrightarrow {\mathpzc{F}}^{r}_{p,q_1}\hookrightarrow {\mathpzc{F}}^{r}_{p,q_2}\hookrightarrow \mathpzc{F}^{r}_{p,\infty},$  $\varepsilon>0,$ $0<p\leq \infty,$ $0<q_{1}\leq q_2\leq \infty.$
\item[(2)]  $\mathpzc{F}^{r+\varepsilon}_{p,q_1}\hookrightarrow \mathpzc{F}^{r}_{p,q_2}$, $\varepsilon>0,$ $0<p\leq \infty,$ $1\leq q_2<q_1<\infty.$
\item[(3)] $\mathpzc{F}^0_{2,2}(\mathbb{R}^{2n})=L^2(\mathbb{R}^{2n})$ and consequently, for every $s\in\mathbb{R},$ ${W}^{s,\,\mathcal{L}}(\mathbb{R}^{2n})=\mathpzc{F}^s_{2,2}(\mathbb{R}^{2n}).$  
\end{itemize}
With the notations above in the next section we prove our main result.

\section{Regularity properties}\label{Regularity properties}

In this section, the space $L^2_{\textnormal{fin}}(\mathbb{R}^{2n})$ consists of those finite linear combinations of special Hermite functions on $\mathbb{R}^{2n}.$

\begin{remark}\label{remark} The set,  $L^2_{\textnormal{fin}}(\mathbb{R}^{2n})$ is a dense subspace of every space
$W^{s,\,\mathcal{L}}{(\mathbb{R}^{2n})},$ this can be proved by using the Parseval identity applied to the special Hermite functions. An argument of convergence of series shows that $L^2_{\textnormal{fin}}(\mathbb{R}^{2n})$ is a dense subspace of every space $\mathpzc{F}^r_{p,2}(\mathbb{R}^{2n}).$
\end{remark}
We will use the previous remark  in the following result.
\begin{lemma}
Let us consider $f\in \mathpzc{F}^0_{p,2} (\mathbb{R}^{2n} )$, then for all $1\leq p\leq \infty$  we have 
\begin{equation}\label{eql2lp}
\Vert u(t,z) \Vert_{L^p_z[\mathbb{R}^{2n},\,L^2_t[0,2\pi]]}=\sqrt{2\pi}\Vert f \Vert_{\mathpzc{F}^0_{p,2}(\mathbb{R}^{2n})}.
\end{equation}
\end{lemma}
\begin{proof}
Let us consider $f\in \mathpzc{F}^0_{p,2}(\mathbb{R}^{2n}).$ By a standard argument of density we only need to assume $f\in L^2_{\textnormal{fin}}(\mathbb{R}^{2n}).$ The solution $u(t,z)$ for \eqref{SEq} is given by
\begin{equation}
u(t,z)=\sum_{\mu,\nu\in\mathbb{N}^n_0}e^{-it(2|\nu|+n)}\widehat{f}(\phi_{\mu\nu})\phi_{\mu\nu}(z)=\sum_{\ell}e^{-it\ell}P_{\ell}f(z).
\end{equation}
Let us note that the previous sums runs over a finite number of $\ell's$ because $f\in L^2_{\textnormal{fin}}(\mathbb{R}^{2n}).$ So, we have,
\begin{align*}
\Vert u(t,z) \Vert_{L^2_t[0,2\pi]}^2&=\int_{0}^{2\pi}u(t,z)\overline{u(t,z)}dt\\
&=\int_{0}^{2\pi} \sum_{\ell,\ell'}e^{-it(\ell-\ell')}P_{\ell}f(z)\overline{ P_{\ell'}f(z)  }  dt \\
&= \sum_{\ell,\ell'}\int_{0}^{2\pi}e^{-it(\ell-\ell')}dtP_{\ell}f(z)\overline{ P_{\ell'}f(z)  } \\&    =  \sum_{\ell} 2\pi P_{\ell}f(z)\overline{ P_{\ell}f(z)  }   \\
&=\sum_{\ell}2\pi\,\cdot|P_{\ell}f(z)|^2, 
\end{align*}
where we have used   the $L^2$-orthogonality of the trigonometric polynomials. Thus, we conclude the following fact
\begin{equation}\label{identity1}
\Vert u(t,z) \Vert_{L^2_t[0,2\pi]}=\left(\sum_{\ell}2\pi\,\cdot|P_{\ell}f(z)|^2\right)^{\frac{1}{2}},\,\,\,  f\in L^2_{\textnormal{fin}}(\mathbb{R}^{2n}).
\end{equation}
Finally,
\begin{equation}
\Vert u(t,z)\Vert_{L^p_{z}(\mathbb{R}^{2n},L^2_t[0,2\pi] )}=\sqrt{2\pi}\Vert f\Vert_{\mathpzc{F}^0_{p,2}(\mathbb{R}^{2n})}.
\end{equation}

\end{proof}

The following lemma, allows us to compare the $L^p_{z}(L^q_{t})$-mixed norms of the solution $u(t,z)$ and some Triebel-Lizorkin norms of the initial data.

\begin{lemma}\label{T1}
Let $p,q$ and $s_q$ be such that $0<p\leq \infty,$ $2\leq q<\infty$ and $s_{q}:=\frac{1}{2}-\frac{1}{q}.$ Then
\begin{equation}
C_p'\Vert f\Vert_{\mathpzc{F}^0_{p,2}}\leq \Vert u(t,z)\Vert_{L^p_{z}(\mathbb{R}^{2n}, L^q_{t}[0,2\pi])} \leq C_{p,s}\Vert f\Vert_{\mathpzc{F}^s_{p,2}}, 
\end{equation}
holds true for every $s\geq s_{q}.$ 
\end{lemma}
\begin{proof}
Let us  consider, by a density argument, the initial data $f\in L^2_f(\mathbb{R}^n).$ In order to estimate the norm  $\Vert u(t,z) \Vert_{L^p_z[\mathbb{R}^{2n},\,L^q_t[0,2\pi]]}$ we can use the Wainger Sobolev embedding Theorem:
\begin{equation}
\left\Vert \sum_{\ell\in \mathbb{Z},\ell\neq 0}|\ell|^{-\alpha}\widehat{F}(\ell)e^{-i\ell t}\right\Vert_{L^q[0,2\pi]}\leq C\Vert F\Vert_{L^r[0,2\pi]},\,\,\alpha:=\frac{1}{r}-\frac{1}{q}.
\end{equation}
For $s_q:=\frac{1}{2}-\frac{1}{q}$  we have
\begin{align*}
\Vert u(t,z)\Vert_{L^q[0,2\pi]} &=\left\Vert \sum_{\ell} e^{-it\ell}P_{\ell}f(z)  \right\Vert_{L^q[0,2\pi]} \\
&\leq C\left\Vert \sum_{\ell} \ell^{s_q}e^{-it\ell}P_{\ell}f(z)  \right\Vert_{L^2[0,2\pi]}\\
&=C\left\Vert \sum_{\ell} e^{-it\ell}P_{\ell}[\mathcal{L}^{s_q}f(z)]  \right\Vert_{L^2[0,2\pi]}\\
&=C\left( \sum_{\ell} |P_{\ell}[\mathcal{L}^{s_q}f(z)|^2  \right)^{\frac{1}{2}}\\
&:=T'(\mathcal{L}^{s_q}f)(z).
\end{align*}
 So, we have
\begin{equation}\label{31}
\Vert u(t,z) \Vert_{L^p_z[\mathbb{R}^{2n},\,L^q_t[0,2\pi]]}\leq C\Vert T'(H^{s_q}f) \Vert_{L^p(\mathbb{R}^{2n})}= C\Vert \mathcal{L}^{s_q}f\Vert_{\mathpzc{F}^0_{p,2}(\mathbb{R}^{2n})}=C\Vert f\Vert_{\mathpzc{F}^{s_q}_{p,2}(\mathbb{R}^{2n})}.
\end{equation}
We end the proof by taking into account the embedding $\mathpzc{F}^s_{p,2}\hookrightarrow \mathpzc{F}^{s_q}_{p,2}$ for every $s>s_q$ and the following inequality for $2\leq q<\infty$
\begin{align*}
\Vert f\Vert_{\mathpzc{F}^0_{p,2}}&\lesssim  \Vert T'f\Vert_{L^p}\\
&= C \Vert u(t,z) \Vert_{L^p_z[\mathbb{R}^{2n},\,L^2_t[0,2\pi]]}\lesssim \Vert u(t,z) \Vert_{L^p_z[\mathbb{R}^{2n},\,L^q_t[0,2\pi]]}.
\end{align*}
\end{proof}

Now, we present the main tool in our further analysis (for the proof we refer the reader to  H. Koch and F. Ricci\cite{KochRicci2007}, the references \cite{Rat1997,Ste1998} include some results for $p\neq \frac{4n+2}{2n-1}$).

\begin{theorem}\label{mainlemmata}
Let $n\in \mathbb{N},$ and $2\leq p\leq \infty.$ Let us consider the orthogonal projection $P_{\ell},$ $\ell\in 2\mathbb{N}+n.$ Then we have
\begin{equation}
    \Vert P_{\ell}f \Vert_{L^p(\mathbb{R}^{2n})}\leq C\ell^{\varkappa_p}\Vert f\Vert_{L^2(\mathbb{R}^{2n})},
\end{equation}where $\varkappa_p=\frac{1}{2}(\frac{1}{p}-\frac{1}{2})$ for $2\leq p\leq \frac{4n+2}{2n-1}$ and $\varkappa_p=n(\frac{1}{2}-\frac{1}{p})-\frac{1}{2}$ for $\frac{4n+2}{2n-1}\leq p\leq \infty.$ The exponent $\varkappa_p$ is the best possible.

\end{theorem}
Now, with the analysis developed above and by using Theorem \eqref{mainlemmata} we will provide a short proof for our main result.
\begin{proof}[Proof of Theorem \ref{MAINtHEO'}]
We observe that from Lemma \ref{T1}, $$ C'\Vert f\Vert_{\mathpzc{F}^0_{p,2}}\leq \Vert u(t,z)\Vert_{L^p_{z}(\mathbb{R}^{2n}, L^q_{t}[0,2\pi])} \leq C_{p,s}\Vert f\Vert_{\mathpzc{F}^s_{p,2}},  $$ so, we only need to estimate the $\mathpzc{F}^s_{p,2}$-norm of the initial data $f,$ by showing that $\Vert f\Vert_{\mathpzc{F}^s_{p,2}}\lesssim \Vert f\Vert_{W^{s,\,\mathcal{L}}},$ for every $s\geq\varkappa_{p,q}. $ Moreover, by the embedding $W^{s,\,\mathcal{L}}\hookrightarrow W^{\varkappa_{p,q},\,\mathcal{L}}$ for every $s\geq \varkappa_{p,q}$ we only need to check the previous estimate when $s=\varkappa_{p,q}.$ By the condition $2\leq p<\infty,$  together with  the Minkowski integral inequality and Theorem \ref{mainlemmata}, we  give
\begin{align*}
\Vert f\Vert_{\mathpzc{F}^{s_q}_{p,2}} &=\left\Vert \left( \sum_{\ell} {\ell}^{2s_q}|P_{\ell}f(z)|^2  \right)^{\frac{1}{2}}\right\Vert_{L^p}\\&\leq \left( \sum_{\ell} {\ell}^{2s_q}\Vert P_{\ell}f\Vert^2_{L^p}  \right)^{\frac{1}{2}}\\
&\lesssim \left( \sum_{\ell} {\ell}^{2(s_q+\varkappa_p)}\Vert P_{\ell}f\Vert^2_{L^2}  \right)^{\frac{1}{2}}\\
&=\Vert f\Vert_{W^{\varkappa_{p,q},\,\mathcal{L}}},
\end{align*}
where we have used that $\varkappa_{p,q}=s_q+\varkappa_p.$
Let us note that the previous estimates are valid for $p=\infty.$ The sharpness of the result for $q=2$ can be proved by choosing $f=P_{\ell'}g$ for some arbitrary, but non-trivial  function $g\in L^2.$ In this case, $\varkappa_{p,q}=\varkappa_p$ because $s_q=0,$ and we have 
$$  \Vert u(t,z) \Vert_{L^{p}_z[\mathbb{R}^{2n},\,L^2_t[0,2\pi]]}=\Vert P_{\ell'}g\Vert_{L^p}\leq \ell'^{\varkappa_p}\Vert P_{\ell'} g\Vert_{L^2}=\Vert g\Vert_{W^{\varkappa_p,\,\mathcal{L}}}.$$ The sharpness of this inequality is consequence of the sharpness for the exponent $\varkappa_p$ in Theorem \ref{mainlemmata}. Indeed,  an improvement of the previous estimate
would yield improved estimates for the spectral projection operator $P_{\ell'}$, which is not
possible. Thus, we finish the proof.
\end{proof}
\begin{remark} In relation with the work by Zhang
and Zheng \cite{ZhangZheng2010}, where the non-homogeneous Schr\"odinger equation associated to $\mathcal{L},$ has been studied,  the  the initial data $f\in L^2$ provides nice  Strichartz estimates as those considered in Keel and Tao \cite{KeelTao}. The analysis employed in \cite{ZhangZheng2010}  comes from suitable estimates for the operators $e^{-it\mathcal{L}},$ $t>0.$ In the homogeneous case \eqref{SEq}, in view of the embedding,
$L^2\hookrightarrow W^{\varkappa,\,\mathcal{L}}$ for $\varkappa<0,$ we have showed that the usual condition $f\in L^{2}$ can be improved if we  consider those cases mentioned in Remark \ref{r1} where $\varkappa_{p,q}<0.$ \end{remark}
\begin{remark}
On the other hand, by following
Bongioanni and Rogers \cite[Section 5]{BonRog}, Theorem \ref{MAINtHEO'} could provide similar estimates to the one obtained in \cite[Theorem 5.1]{BonRog}, for the Scr\"odinger equation associated to the forced twisted Laplacian $\mathcal{L}+V(z,t),$ where the potential $V(z,t),$ could be considered small enough  in some sense. This could be obtained from a routine argument via Duhamel's formula together with the contraction mapping principle.
\end{remark}

\textbf{Acknowledgment.} This work was supported by Pontificia Universidad Javeriana, Department of Mathematics. The author would like to kindly thank the reviewers of this manuscript for their detailed comments which helped to present this final version.

\end{document}